\documentclass[a4paper, svgnames, 10pt]{amsart}

\usepackage[a-1b]{pdfx}

\usepackage{
    amssymb,
    mathtools,
    xcolor,
    booktabs,
}

\usepackage{enumitem}
\setlist[enumerate]{itemsep=0.25\baselineskip, font=\normalfont}

\usepackage[english]{babel}

\usepackage{hyperref}
\hypersetup{colorlinks, allcolors=MediumBlue}
\newcommand{\msc}[1]{\href{https://zbmath.org/classification/?q=#1}{#1}}
\newcommand{\doi}[1]{doi:\href{https://doi.org/#1}{#1}}
\usepackage[noabbrev, capitalise]{cleveref}
\crefdefaultlabelformat{#2\textup{#1}#3}

\crefname{main} {Theorem}       {Theorems}
\crefname{thm}  {Theorem}       {Theorems}
\crefname{lem}  {Lemma}         {Lemmas}
\crefname{prop} {Proposition}   {Propositions}
\crefname{dfn}  {Definition}    {Definitions}
\crefname{fig}  {Figure}        {Figures}
\crefname{tbl}  {Table}         {Tables}
\crefname{rmk}  {Remark}        {Remarks}
\crefname{exm}  {Example}       {Examples}
\crefname{que}  {Question}      {Questions}

\theoremstyle{plain}
\newtheorem{thm} {Theorem} [section]
\newtheorem{prop}   [thm] {Proposition}
\newtheorem{cor}    [thm] {Corollary}
\newtheorem{lem}    [thm] {Lemma}
\theoremstyle{definition}

\newtheorem{exm}    [thm] {Example}
\newtheorem{que}    [thm] {Question}

\numberwithin{equation}{section}

\newcommand*{\case}[1]{\paragraph{\indent\textit{#1}}}

\DeclarePairedDelimiterX\set[1]\lbrace\rbrace{\,\def\given{\mid}#1\,}
\DeclarePairedDelimiterX\gen[1]\langle\rangle{\,\def\given{\mid}#1\,}
\DeclareMathOperator{\Sub}{Sub}              
\DeclareMathOperator{\Cl}{Cl}                
\newcommand{\Agemo}{\mho}                    
\DeclareMathOperator{\Frat}{\Phi}            

\begin{document}

\title{The Modular Isomorphism Problem over all fields}

\author[L.~Margolis]{Leo Margolis
}
\address[Leo Margolis]
{Universidad Aut\'onoma de Madrid, Departamento de Matem\'aticas,  C/ Francisco Tom\'as y Valiente 7, Facultad de Ciencias, m\'odulo 17, 28049 Madrid, Spain.}
\email{leo.margolis@icmat.es}

\author[T.~Sakurai]{Taro Sakurai
}
\address[Taro Sakurai]
{Department of Mathematics and Informatics, Graduate School of Science, Chiba University, 1-33, Yayoi-cho, Inage-ku, Chiba-shi, Chiba, 263-8522, Japan.}
\email{tsakurai@math.s.chiba-u.ac.jp}



\date{\today}

\begin{abstract}
    The Modular Isomorphism Problem asks, if an isomorphism between modular group algebras of finite $p$-groups over a field $F$ implies an isomorphism of the group bases.
    We explore the differences of knowledge on the problem when $F$ is either assumed to be a prime field or a general field of characteristic~$p$.
    After revising the literature and explaining reasons for the differences, we generalize some of the positive answers to the problem from the prime field case to the general case.
\end{abstract}

\maketitle


\section{Introduction}

The Modular Isomorphism Problem (MIP) asks, if the isomorphism type of a finite $p$-group $G$ can be recovered from the isomorphism type of the group algebra $FG$ over a field $F$ of characteristic~$p$, when $FG$ is viewed as an $F$-algebra.
The problem appears in the literature in two forms, the more common being:

\begin{description}
    \item[MIP over prime fields]
          Let $G$ be a finite $p$-group and $\mathbb{F}_p$ the field of $p$~elements.
          Does an isomorphism of group algebras $\mathbb{F}_p G \cong \mathbb{F}_p H$ for some group $H$ imply an isomorphism of groups $G \cong H$?
\end{description}

Sometimes though, e.g.\ in~\cite[Conjecture 9.4]{ZalesskiiMihalev73} or~\cite{Deskins56, Baginski92, Sandling96, BleherKimmerleRoggenkampWursthorn99}, it appears in the more general form on which we focus here:

\begin{description}[resume]
    \item[MIP over all fields]
          Let $G$ be a finite $p$-group and $F$ a field of characteristic~$p$.
          Does an isomorphism of group algebras $FG \cong FH$ as $F$-algebras for some group $H$ imply an isomorphism of groups $G \cong H$?
\end{description}

Though negative solutions to the Modular Isomorphism Problem in case $p=2$ have been found a few years ago~\cite{GarciaLucasMargolisDelRio22} and recently slightly generalized~\cite{MargolisSakurai25, BaginskiZabielski25}, those groups offer negative solutions for both formulations.
In particular, it remains unknown if for a fixed $G$ the two formulations of the problem are equivalent.
The first formulation is stronger than the second as tensoring $\mathbb{F}_p G \cong \mathbb{F}_p H$ with $F$ over $\mathbb{F}_p$ yields $FG \cong FH$ by \cite[Lemma~1.3.6]{Passman77}.
It is known though, that in the formulation of MIP over all fields it is sufficient to consider only finite fields~\cite{GarciaLucasDelRio24}.

The aim of this note is to first summarize the results available for MIP over all fields and provide a list of invariants which can be useful when studying the problem.
We then point out the differences between the results available for the two formulations and some reasons for them.
This includes an example of two algebras, given by modular group algebras modulo a power of augmentation ideals, which are non-isomorphic over a prime field but become isomorphic over a certain extension field.
Finally, we generalize some results from the literature which solve MIP over prime fields to show that they hold over any field (viz.,~\cref{thm:Metacyclic,thm:BaginskiKurdics,th:BrocheDelRio}).
The idea that this could be an interesting goal was also pointed out by Sandling, as recorded in Borge's thesis~\cite[p.~146]{Borge01}.

Our group-theoretical notation is mostly standard.
For a group $G$ we denote by $G'$ its derived subgroup and by $\gamma_n(G)$ the $n$-th term of the lower central series of $G$.
By $Z(G)$ we denote the center of $G$ and by $\Frat(G)$ the Frattini subgroup of $G$.
For elements $g, h \in G$ the centralizer of $g$ in $G$ is denoted by $C_G(g)$ and the commutator $g^{-1}h^{-1}gh$ is denoted by $[g, h]$.
A cyclic group of order $n$ is denoted by $C_n$.
For a non-negative integer $k$ set \[ \Agemo_k(G) = \gen[\big]{ g^{p^k} \given g \in G }, \ \ \Omega_k(G) = \gen[\big]{ g \in G \given g^{p^k} = 1 }.
\]
For a normal subgroup $N$ of $G$ set
\[
    \Omega_k(G:N) = \gen[\big]{ g \in G \given g^{p^k} \in N }.
\]
We will also write $\Agemo(G)$ for $\Agemo_1(G)$.

For group algebras we denote by $\Delta(FG)$ the augmentation ideal of $FG$ which in our case coincides with the Jacobson radical.
We write the $n$-th power of the augmentation ideal as $\Delta^n(FG)$.
By $[FG,FG]FG$ we denote the ideal generated by the Lie commutators $\set{ ab-ba \given a,b \in FG }$.
For a subgroup $U \leq G$ the relative augmentation ideal of $U$ in $FG$ is the ideal generated by $\set{ u-1 \given u \in U }$ and when $U$ is normal in $G$ we denote it by $\Delta(U)FG$.
We note that in our case $\Delta(G')FG = [FG,FG]FG$.
Moreover, we recall that if $U$ is normal, we have a natural isomorphism $FG/\Delta(U)FG \cong F[G/U]$ of $F$-algebras.

Throughout the whole article $G$ denotes a finite $p$-group and $F$ a field of characteristic~$p$.
All fields are assumed to be of characteristic $p > 0$.

\section{Known results}
\subsection{Classes and invariants}

We collect the classes for which the MIP is known over all fields and also the invariants which are known in this case.
It turns out that a majority of the results available concerns the case of $2$-groups only, and also that many early results on the MIP concern only the prime field case and do not enter here.

Parts of the following compilation appeared also in~\cite{Margolis22}.

\begin{thm}\label{th:KnownClasses}
    Assume $G$ is a finite $p$-group, $F$ a field of characteristic~$p$ and $H$ a group such that $FG \cong FH$.
    If $G$ is among the groups in the following list, then $G \cong H$.
    \begin{enumerate}
        \item Groups with $[G:Z(G)] \leq p^2$ \cite{Drensky89}, \label{it:Drensky}
        \item two-generated groups of nilpotency class two for $p$ odd \cite{BrocheDelRio21}\footnote{The result is stated only over prime fields, but the arguments applied in the case of odd $p$ work for any field.
                  For~$p = 2$, see~\cref{th:BrocheDelRio}.
              },
        \item Groups of nilpotency class two with cyclic center and either $p$ is odd or $p=2$ and for $G/Z(G) \cong C_{2^{m_1}} \times \dotsb \times C_{2^{m_k}}$ with $m_1 \geq m_2 \geq \dotsb \geq m_k$ one has $m_1 > m_3$ \cite{GarciaLucasMargolis24}.
    \end{enumerate}

    Moreover, for $p=2$ in the following cases the statement also holds.
    \begin{enumerate}[resume]
        \item $2$-groups of maximal class \cite{Baginski92},
        \item Groups of order $32$ \cite[Lemma 3.7]{NavarroSambale18},
        \item $2$-groups of nilpotency class three such that $[G:Z(G)] = |\Frat(G)| = 8$ \cite{MargolisSakuraiStanojkovski23},
        \item $2$-groups with cyclic center such that $G/Z(G)$ is dihedral \cite{MargolisSakuraiStanojkovski23},
        \item two-generated $2$-groups such that $G/Z(G)$ is non-abelian dihedral with the exception of known counterexamples to the MIP \cite{MargolisSakurai25}.
    \end{enumerate}
\end{thm}

Subclasses of the groups mentioned in the previous theorem were also considered in the literature and the MIP was proven there over all fields~\cite{Deskins56, Coleman64, NacevMollov81}.
We notice that for small $2$-groups MIP over all fields is hence known until order 32, as groups of order at most 16 either have a center of index at most $p^2$ or are of maximal class.

We next summarize the group-theoretical invariants available over all fields.
Here, for a field $F$, a property of $G$ or $FG$ is called \emph{$F$-invariant}, if an isomorphism of $F$-algebras $FG \cong FH$ implies the same property for $H$ or $FH$.
E.g.\ the order of~$G$ is the dimension of~$FG$ as an $F$-vector space and as such an $F$-invariant.
Many of the invariants have been known for decades, though we will not include original references for all of them.

The most productive source to obtain $F$-invariants has been to consider ideals or subalgebras of $FG$ defined independently of the chosen basis $G$ and then to consider powers, quotients and intersections involving these objects.
E.g.\ the most classical $F$-invariant has its roots in the work of Jennings~\cite{Jennings41} and expresses the dimensions of quotients of subsequent powers of the Jacobson radical of $FG$ as certain sections of $G$.
Namely Jennings showed that the dimension subgroup \[ D_n(G) = G \cap (1 + \Delta^n(FG)) \] can be expressed in purely group-theoretical terms which was later translated by Lazard to the expression $D_n(G) = \prod_{ip^j \geq n} \Agemo_j(\gamma_i(G))$, see~\cite[Theorem~11.1.20]{Passman77}.
Moreover, Jennings also showed that the isomorphism type of $D_n(G)/D_{n+1}(G)$ is an $F$-invariant.
This is one of the basic invariants included already in Passman's classical book on the subject.

\begin{prop}[{\cite[Lemma 14.2.7]{Passman77}}]\label{prop:BasicInvariants}
    The following are $F$-invariants of $FG$.
    \begin{enumerate}
        \item The isomorphism type of $G/G'$.
        \item The isomorphism type of $Z(G)$.
        \item The isomorphism type of $D_n(G)/D_{n+1}(G)$ for any $n \geq 1$.
    \end{enumerate}
\end{prop}

We note that $D_2(G) = \Frat(G)$ and hence the minimal number of generators of $G$, which equals the rank of $G/\Frat(G)$, is an $F$-invariant.

Several more $F$-invariants can be derived from these type of arguments as demonstrated for example in \cite{Sandling85, MargolisSakuraiStanojkovski23}, but most, if not all, of them were generalized in an elegant way by Garc\'ia-Lucas in \cite{GarciaLucas24}.
The main ingredient in his results is the following.

\begin{thm}[{``Transfer Lemma'' \cite[Lemma 3.6]{GarciaLucas24}}]\label{th:TransferLemma}
    Let $\varphi\colon FG \to FH$ be an $F$-algebra isomorphism.
    Moreover, let $N_G$, $L_G$ be normal subgroups of $G$ and $N_H$, $L_H$ normal subgroups of $H$ such that $N_G$ and $N_H$ contain the derived subgroup of $G$ and $H$, respectively.
    Assume that \[ \text{$\varphi(\Delta(N_G)FG) = \Delta(N_H)FH$ and $\varphi(\Delta(L_G)FG) = \Delta(L_H)FH$.
        }
    \]
    Then for each non-negative integer $k$ the following equalities hold.
    \begin{enumerate}
        \item $\varphi(\Delta(\Omega_k(G:N_G))FG) = \Delta(\Omega_k(H:N_H))FH$.     
        \item $\varphi(\Delta(\Agemo_k(L_G)N_G)FG) = \Delta(\Agemo_k(L_H)N_H)FH$.   
        \item $\varphi(\Delta(\Omega_k(Z(G))N_G)FG) = \Delta(\Omega_k(Z(H))N_H)FH$. 
    \end{enumerate}
\end{thm}

Together with the following this will allow us to find many $F$-invariants.

\begin{prop}\label{prop:TransferLemmaApplications}
    Assume the notation of \cref{th:TransferLemma}.
    Moreover, let $M_G$ and $M_H$ be central subgroups of $G$ and $H$, respectively.
    Assume that \[ \varphi(FM_G \oplus (Z(FG) \cap [FG, FG])) = FM_H \oplus (Z(FH) \cap [FH, FH]).
    \]
    Then the following isomorphisms hold:
    \begin{enumerate}
        \item $D_n(L_G)/D_{n+1}(L_G) \cong D_n(L_H)/D_{n+1}(L_H)$ for every positive integer $n$.
        \item $G/L_GN_G \cong H/L_HN_H$ and $L_GN_G/N_G \cong L_HN_H/N_H$.
        \item $M_G \cap L_G \cong M_H \cap L_H$ and $M_G/(M_G \cap L_G) \cong M_H/(M_H \cap L_H)$.
    \end{enumerate}
\end{prop}
\begin{proof}
    All except the last can be already found in~\cite[Lemma 3.2]{GarciaLucas24}.
    The last statement for $M_G = Z(G)$ can be also found in~\cite[Lemma 3.2(5)]{GarciaLucas24}, but we take this opportunity to record a slightly more general statement.
    The third is ``dual'' to the second in some sense;
    see~\cref{fig:subgroups}.
    \begin{figure}[htbp]
        \includegraphics{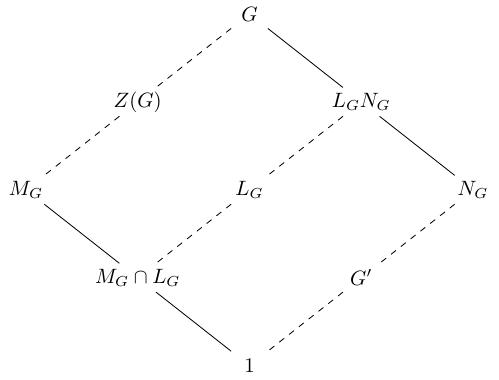}
        \caption{
            Subgroups in \cref{prop:TransferLemmaApplications}.
            The solid lines depict obtained $F$-invariants.
        }
        \label{fig:subgroups}
    \end{figure}

    Since the proof goes in the same way as in~\cite[\S~2.2]{MargolisSakuraiStanojkovski23}, we only provide a rough outline here.
    The main idea is to consider sequences \[ M_G \geq M_G \cap L_G \geq \Agemo_k(M_G \cap L_G) \geq 1, \] \[ FM_G \geq \Delta(M_G \cap L_G)FM_G \geq \Delta(\Agemo_k(M_G \cap L_G))FM_G \geq 0 \] for $k \geq 0$, which corresponds to~\cite[Fig.~3]{MargolisSakuraiStanojkovski23}.
    Technically speaking, we consider an algebra $FM_G \oplus (Z(FG) \cap [FG, FG])$ and an ideal \[ \Theta(FG) = \Delta(\Agemo_k(M_G \cap L_G))FM_G \oplus (Z(FG) \cap [FG, FG]), \] which corresponds to~\cite[(2.4)]{MargolisSakuraiStanojkovski23}.
    Then $\varphi(\Theta(FG)) = \Theta(FH)$ from assumptions since $\Theta(FG)$ can be written as
    \[
        (\Delta(M_G \cap L_G)FM_G \oplus (Z(FG) \cap [FG, FG]))^{p^k}FM_G + (Z(FG) \cap [FG, FG]),
    \]
    which corresponds to~\cite[(2.2)]{MargolisSakuraiStanojkovski23}.
    From
    \[
        F[M_G/\Agemo_k(M_G \cap L_G)]
        \cong (FM_G \oplus (Z(FG) \cap [FG, FG]))/\Theta(FG),
    \]
    we get $M_G/(M_G \cap L_G) \cong M_H/(M_H \cap L_H)$ by setting $k = 0$ and \cref{th:KnownClasses}(\ref{it:Drensky}).
    Since $|\Agemo_k(M_G \cap L_G)| = |\Agemo_k(M_H \cap L_H)|$ for all $k \geq 0$, we have $M_G \cap L_G \cong M_H \cap L_H$ by \cite[Proposition~1.1]{MargolisSakuraiStanojkovski23}.
\end{proof}

Of course to apply \cref{th:TransferLemma} and \cref{prop:TransferLemmaApplications} one needs a starting choice for $M_G$, $N_G$ and $L_G$.
For $N_G$ and $L_G$, clearly $G$ can be one of these subgroups and as $\Delta(G')FG$ equals the ideal of $FG$ generated by Lie commutators also $G'$ has the needed property, which in \cite{GarciaLucas24} is called a canonical assignation.
Starting with those choices one can then obtain many $F$-invariants iteratively, some of which we list in the following corollary.

\begin{cor}[{\cite[Example 3.7(3)]{GarciaLucas24}}]\label{cor:TransferLemmaInvariants}
    The isomorphism types of the following abelian sections of $G$ are $F$-invariants for every non-negative integer $k$.
    \begin{enumerate}
        \item $Z(G) \cap \Agemo_k(G)G'$ and $Z(G)\Agemo_k(G)G'/\Agemo_k(G)G'$.
        \item $G/\Agemo_k(Z(G))G'$ and $\Agemo_k(Z(G))G'/G'$.
        \item $G/\Omega_k(Z(G))G'$ and $\Omega_k(Z(G))G'/G'$.
    \end{enumerate}
    In particular, $Z(G) \cap G'$, $Z(G)G'/G'$, $G/Z(G)G'$ and $Z(G)G'/G'$ are $F$-invariants.
\end{cor}

The previous corollary does not exceed the possible applications of \cref{th:TransferLemma} and \cref{prop:TransferLemmaApplications}.
E.g.\ we will use the first points of \cref{th:TransferLemma} and \cref{prop:TransferLemmaApplications} in the proof of \cref{th:BrocheDelRio}.
Also the following structural property can be seen as a consequence of the above results.

\begin{prop}[{\cite[Theorem 4.1]{MargolisSakuraiStanojkovski23}}]
    Assume $FG \cong FH$ and that $G = A \times U$ and $H = B \times V$ such that $A$ and $B$ are elementary abelian and no larger elementary abelian direct factors of $G$ and $H$ exist, respectively.
    Then $A \cong B$ and $FU \cong FV$.
\end{prop}

Though a similar theorem was shown to hold for abelian direct factors over prime fields in \cite{GarciaLucas24}, it is not known if this generalizes to all fields.

In a somehow similar spirit to performing algebraic operations to get invariant sections, but considering the dimension of certain images, the following numerical invariants involving conjugacy classes can be obtained.
These are due to K\"ulshammer and Parmenter--Polcino Milies, respectively.
We refer to \cite[p.~13]{HertweckSoriano06} for short proofs.
We will write $\Cl(G)$ for the set of conjugacy classes of a group $G$ and $C^n = \set{ g^n \given g \in C }$ for the $n$-th power of a conjugacy class $C \in \Cl(G)$.

\begin{prop}\label{prop:InvariantsCC}
    For each non-negative integer~$k$, the sizes of the following two sets are $F$-invariants.
    \begin{equation*}
        \set[\big]{ C^{p^k} \given C \in \Cl(G) }, \  \set[\big]{ C^{p^k} \given C \in \Cl(G),\ |C^{p^k}| = |C| }.
    \end{equation*}
    In particular, the exponent of $G$ is an $F$-invariant.
\end{prop}

Another fruitful source of $F$-invariants is provided by cohomology.
As cohomology rings are defined in their abstract form in terms of the group algebra alone, if one has an $FG$-module $M$ which is defined independently of the given basis $G$, then its cohomology ring is an $F$-invariant.

This includes the cohomology ring $H^*(G, F)$ of the trivial module $F$ as well as the Hochschild cohomology $HH^*(FG)$ which is the cohomology ring of the group algebra considered as a module over its enveloping algebra.
Of course, describing cohomology rings is in general not easy, but some data can also be given purely in group-theoretical terms.
For instance, the dimension of $H^1(G,F)$ is the minimal number of generators of $G$, which we already know to be an $F$-invariant as mentioned above.
Deeper group-theoretical invariants follow from Quillen's stratification of $H^*(G, F)$ \cite{Quillen71a, Quillen71b} and from an expression of $\dim HH^1(FG)$ (cf.~\cite[Theorem 2.11.2]{Benson98b} and \cite[Proposition 3.14.2]{Benson98a}):

\begin{prop}\label{prop:CohomInvariants}
    For each non-negative integer~$k$, the number of conjugacy classes of maximal elementary abelian subgroups of $G$ of rank~$k$ is an $F$-invariant.

    Moreover, the following value, which equals the dimension of the first Hochschild cohomology group of $FG$, is also an $F$-invariant \[ \sum_{g^G} \dim H^1(C_G(g), F) = \sum_{g^G} \log_p |C_G(g)/\Frat(C_G(g))|, \] where the sums run over all the conjugacy classes $g^G$ of $G$.
\end{prop}

The maximal elementary abelian subgroups can be used to see for instance that dihedral groups and generalized quaternion groups have non-isomorphic group algebras over any field.
We will use the dimension of the first Hochschild cohomology group for the proof of \cref{thm:BaginskiKurdics} below.

There are several fundamental properties of $G$ for which it is unknown whether they are $F$-invariants.
This includes the nilpotency class of~$G$.
We summarize what is known here, referring to \cite[Theorem 2]{BaginskiKonovalov07} for proofs of the first three points and the formulation of \cite[Theorem 3.1]{BaginskiKurdics19} together with \cref{prop:BasicInvariants} makes clear that \cite[Corollary 3.1]{BaginskiKurdics19} holds over any field.

\begin{prop}\label{prop:Class}
    The nilpotency class of $G$ is an $F$-invariant, if one of the following holds:
    \begin{enumerate}
        \item The exponent of $G$ is $p$.
        \item $G'$ is cyclic.
        \item $G$ is of nilpotency class two.
        \item $G$ is of maximal class.
    \end{enumerate}
\end{prop}

We would like to mention that all of the invariants described above can be explored in the \mbox{\textsf{ModIsom}} package~\cite{ModIsom} of \textsf{GAP}~\cite{GAP} and in particular functions are available to apply them one after the other to a given set of groups.

To close this section, we would like to mention an idea presented in~\cite{Sakurai20} which allows to study MIP over all fields when combined with a reduction~\cite{GarciaLucasDelRio24}.
The unit groups of algebras, not necessarily of group algebras nor their quotients, play a key role there and it would be interesting to seek similar or alternative constructions of algebras~\cite[\S\S~3.2--3.3]{Sakurai20} that work over all finite fields.

\subsection{Differences between prime fields and all fields}

We notice that the results in \cref{th:KnownClasses} do not cover several rather large classes of groups for which a positive solution to MIP over prime fields is known, such as groups of nilpotency class two with elementary abelian derived subgroup \cite{Sandling89} or groups with trivial fourth dimension subgroup when $p$ is odd \cite{Hertweck07}.
Neither does it include the groups of order $p^4$, one of the earliest results about the MIP~\cite{Passman65}.
There are essentially three reasons why these classes are not covered and we shortly revise the main differences in the techniques available to study MIP over prime fields or over all fields.

\case{Difference 1}
One classical idea to study isomorphism types of modular group algebras consists in counting elements which lie in the kernel of certain maps defined on sections of the group algebra, typically power maps on quotients of powers of the augmentation ideal.
The resulting numbers are sometimes called kernel sizes.
This idea was already suggested by Brauer in \cite{Brauer63} and carried out systematically for the first time by Passman~\cite{Passman65}, cf.~also \cite[\S~2.7]{HertweckSoriano06}.
While the idea is feasible over any field, it might happen that the kernel sizes are different over prime fields but become equal over certain extension fields.
For example, the last argument in~\cite{Passman65} involves the fact that if $\alpha$ is a quadratic non-residue modulo $p$, then the polynomial $\alpha X^2 - 1$ has no root in $\mathbb{F}_p$, which is of course incorrect in $\mathbb{F}_{p^2}$.
Also, in \cite[Theorem 1.2]{GarciaLucasMargolis24} the restrictions on the field is due to such a fact.

On the other hand, the technique has also been used to obtain results over all fields, e.g.\ it was used by Drensky \cite[Proposition 2.3]{Drensky89} or in \cite[Lemma 4.8]{MargolisSakurai25}.
In both cases it is less about a size though, but about the existence of non-trivial elements in the kernel.
Overall, kernel sizes are hence still useful to study MIP over all fields, but their application is sometimes more difficult and the achievable results are in general weaker than over prime fields.
We mention that the \mbox{\textsf{ModIsom}} package \cite{ModIsom} admits the calculation of the kernel size for power maps from version 3.0.0 by means of the function \texttt{KernelSizePowerMap}.

\case{Difference 2}
One of the early results on MIP over prime fields was the proof of Passi and Sehgal that the isomorphism type of $D_n(G)/D_{n+2}(G)$ is an $\mathbb{F}_p$-invariant for every positive integer $n$ and so in particular MIP over prime fields has a positive answer for groups with trivial third dimension subgroup \cite{PassiSehgal72}.
To see why their arguments require the field of $p$ elements define iteratively the series of Lie ideals as $\Delta^{[1]}(FG) = \Delta(FG)$ and $\Delta^{[n+1]}(FG) = [FG,\Delta^{[n]}(FG)]FG$ for each positive integer $n$.
Then the $n$-th Zassenhaus ideal of $FG$ is defined as $Z_n(FG) = \sum_{ip^j \geq n} \Delta^{[i]}(FG)^{p^j} + \Delta^{n+1}(FG)$.
The fact realized in \cite{PassiSehgal72} is that for $F=\mathbb{F}_p$ one has an equality $Z_n(FG) = (D_n(G) - 1) +\Delta^{n+1}(FG)$ for any $n$, so the $n$-th dimension subgroup can be identified in $\mathbb{F}_pG$ modulo the $(n+1)$-th power of the augmentation ideal.
In turn, the importance of $\mathbb{F}_p$ is most visible in the more detailed proof given by Sehgal in his book \cite{Sehgal78}, namely in the proof of \cite[Claim (6.15), p.~114]{Sehgal78}:
for $g \in G$ with $g-1 \in \Delta^n(FG)$ and $\lambda \in \mathbb{F}_p$ one has
\begin{equation}\label{eq:CongruenceModHigherDelta}
    \lambda(g-1) \equiv g^\ell-1 \mod \Delta^{n+1}(FG),
\end{equation}
where $\ell \in \mathbb{Z}$ maps to $\lambda$ modulo $p$. 
This has no clear analogue for general fields.

The result of Passi and Sehgal has been later generalized in \cite{Furukawa81, RitterSehgal83} to obtain the isomorphism type of $D_n(G)/D_{2n+1}(G)$ as an $\mathbb{F}_p$-invariant.
Zassenhaus ideals, or rather the observation on them described above, also play a central role in the techniques employed in \cite{HertweckSoriano06, Hertweck07}, so in particular in the result that the isomorphism type of $G/D_4(G)$ is an $\mathbb{F}_p$-invariant when $p$ is odd.
Not the Zassenhaus ideals themselves, but still the congruence \eqref{eq:CongruenceModHigherDelta} also enters into the proof of the main theorem of \cite{GarciaLucas24} which states that in the context of MIP over prime fields direct abelian factors can be canceled.

We would like to remark that all the invariants known today over all fields do not allow to cover the class of groups with trivial third dimension subgroup.
In the given context this is a wide class, but it appears still to be accessible.
This leads us to ask:

\begin{que}
    Does MIP over all fields hold for groups with trivial third dimension subgroup?
\end{que}

\case{Difference 3}
A third method to study MIP over prime fields is the use of the small group ring $\mathbb{F}_pG/\Delta(\mathbb{F}_pG)\Delta(\mathbb{F}_pG')$ and its unit group, which we denote by $S(\mathbb{F}_pG)$.
This was first done by Sandling \cite{Sandling89} who used it to find the isomorphism type of $G/\Agemo(G')\gamma_3(G)$ as an $\mathbb{F}_p$-invariant.
So especially, MIP over prime fields has a positive answer for $G$, if $\Agemo(G')\gamma_3(G)=1$.
The small group ring also played a central role in the proof of MIP over prime fields for groups of order $p^5$ \cite{SalimSandling96} and was later used in \cite{MargolisStanojkovski22, MargolisMoede22, BrennerGarciaLucas24} for some more cases.

These applications are based on an explicit description of an independent generating set of the unit group of $\mathbb{F}_pA$ by Sandling, where $A$ denotes a finite abelian group \cite{Sandling84}.
This description can be used to show that under certain conditions $G$ embeds in $S(\mathbb{F}_pG)$ as a normal subgroup and one can compute explicitly the action of a complement on it.
It might be possible that the method could also bring about some results for MIP over all fields, in particular using the description of an independent generating set of the unit group of a commutative modular group algebra over any field in \cite[Theorem 2.3]{BovdiSzakacs95}.
But it seems that even when $\Agemo(G')\gamma_3(G)=1$, the group $G$ will not embed as a normal subgroup in $S(FG)$, so some new ideas are needed to make the method work over all fields.
Apart from the original paper \cite{Sandling89}, we refer to \cite[\S~2.3]{HertweckSoriano06} and \cite[\S~2.5]{MargolisStanojkovski22} for more details on small group rings in the context of the MIP.

We next give an example, obtained from modular group algebras, that illustrates the differences of considering algebras over prime fields or more general fields.

\begin{exm}
    The smallest non-abelian and non-isomorphic groups of the same order are the dihedral group of order 8, denoted $D_8$, and the quaternion group of order $8$, denoted $Q_8$.
    The group algebras $FD_8$ and $FQ_8$ have been compared many times using different techniques.
    This is all the more interesting because at least the very basic invariants from \cref{prop:BasicInvariants} are the same for both groups.
    Techniques used (usually to prove something more general) include direct calculations of the relations of possible groups bases \cite{Holvoet66}, kernel size \cite[p.~408]{Passman65}, cohomology \cite{Carlson77}, counting elements satisfying certain equations using automorphisms of the group algebra \cite{Baginski92} or ring-theoretical considerations \cite{CreedonHughesSzabo21}.
    Some of these ideas work over any field, while others only work over some fields.
    Here we will show that these algebras are also interesting from the point of view taken in this note.

    Set $\Lambda = \Delta(FD_8)/\Delta^3(FD_8)$ and $\Gamma = \Delta(FQ_8)/\Delta^3(FQ_8)$.
    Then $\Lambda^3=0$ and $\Gamma^3=0$, i.e.\ $\Lambda$ and $\Gamma$ are nilpotent algebras of degree $3$.
    Moreover, $\Lambda$ is generated by two elements $a$ and $b$ which satisfy $ab+ba=a^2$ and $b^2=0$, while $\Gamma$ is generated by $x$ and $y$ which satisfy $xy + yx = x^2$ and $x^2 = y^2$.
    We first show that $\Lambda$ and $\Gamma$ are non-isomorphic when $F = \mathbb{F}_2$.
    Indeed, the only elements with non-zero square in $\Lambda$ are the elements in $a + \Lambda^2$, which can be seen as \[ (a+b)^2 = a^2 + ab + ba + b^2 = a^2 + a^2 = 0, \] while in $\Gamma$ both generators $x$ and $y$, and hence the elements in $x + \Gamma^2$ and $y + \Gamma^2$, have non-zero square.
    So in $\Lambda$ there are 4 such elements, while in $\Gamma$ there are 8.

    Next we show that the algebras are isomorphic when $F = \mathbb{F}_4$.
    For this let $\omega$ be a generator of the unit group of $F$ and note that $\omega^2 + \omega + 1 = 0$.
    In $\Lambda$ consider the elements $x' = a$ and $y' = \omega a + b$.
    We will show that $x'$ and $y'$ satisfy the same relations as $x$ and $y$ which implies in particular that $\Lambda$ and $\Gamma$ are isomorphic in this case.
    Now, \[ (y')^2 = (\omega a + b)^2 = \omega^2 a^2 + \omega(ab + ba) + b^2 = (\omega^2 + \omega)a^2 = a^2 = (x')^2 \] and \[ x'y' + y'x' = a(\omega a + b) + (\omega a + b)a = ab + ba = a^2 = (x')^2.
    \]
    So indeed, the relations of $x$ and $y$ hold for $x'$ and $y'$.
\end{exm}

It is interesting to note that in \cite{Drensky89} Drensky combined most of the techniques described above to prove MIP over all fields in case the center of $G$ has index $p^2$.
Namely, he used the results of Jennings, direct algebraic calculations in the spirit of the kernel size technique involving the ideal $\Delta(G')FG$ and the center of $FG$ and also dimensions of cohomology groups.

\section{From prime fields to all fields}

\subsection{Finite metacyclic $p$-groups}
MIP over prime fields was studied for the class of finite metacyclic $p$-groups first by Bagi\'nski~\cite{Baginski88}, and a complete proof over prime fields was given later by Sandling~\cite{Sandling96}.
In~\cite{Margolis22} it was claimed that the arguments from \cite{Sandling96} show that MIP over all fields holds for metacyclic groups.
This, i.e.\ that the proof generalizes directly, is however not true and was based on a wrong reading of \cite[Theorem 3.1]{Shalev90} which does not in fact hold for $p=2$.
We fill in the missing part of the claim, illustrating in particular the use of ideals.
First we make three easy group-theoretical observations.
We start with a variant of the Correspondence Theorem for finite $p$-groups.
Recall that $\Agemo(G) = \gen{ g^p \given g \in G }$ and $\Frat(G) = \Agemo(G)G'$.
For a normal subgroup $K$ of $G$, the set of subgroups of $G$ containing $K$ is denoted by $\Sub(G : K)$.
\begin{lem}\label{lem:Corresponding}
    Let $K$ and $L$ be normal subgroups of $G$.
    Assume $\Frat(K)$ contains $L$.
    Then the canonical bijection \[ \Sub(G:K) \to \Sub(G/L:K/L) \] is rank-preserving, that is, preserves the minimal number of generators.
\end{lem}
\begin{proof}
    Let $H$ be a subgroup of $G$ and assume $H$ contains $K$.
    Since $K$ is normal, $\Frat(H)$ contains $\Frat(K)$ and, by the assumption, $\Frat(H)$ contains $L$.
    Then $\Frat(H/L) = \Frat(H)/L$ and hence $H/\Frat(H) \cong (H/L)/\Frat(H/L)$, which proves the assertion as the dimensions of Frattini quotients equal ranks.
\end{proof}

\begin{lem}\label{lem:MetacyclicModFratG'}
    $G$ is metacyclic if and only if $G/\Frat(G')$ is metacyclic.
\end{lem}
\begin{proof}
    First, if $G$ is metacyclic, then so is a quotient $G/\Frat(G')$.
    Conversely suppose that $G/\Frat(G')$ is metacyclic.
    Then $G$ has a normal subgroup $N$ containing $\Frat(G')$ such that $N/\Frat(G')$ and $(G/\Frat(G'))/(N/\Frat(G')) \cong G/N$ are both cyclic.
    Observe that $N/\Frat(G')$ contains $G'/\Frat(G')$ and thus $N$ contains $G'$.
    By \cref{lem:Corresponding}, $N$ is also cyclic.
    Hence $G$ is metacyclic.
\end{proof}

\begin{lem}\label{lem:G'HasOrderP}
    Assume $G'$ has order $p$ and $G$ is two-generated.
    Then $G/Z(G)$ is elementary abelian of rank~$2$.
\end{lem}
\begin{proof}
    As $G'$ is normal and of order $p$, it is also central and so the standard commutator formulas imply that $[x,y^p] = [x,y]^p = 1$ for any $x,y \in G$.
    Hence $G/Z(G)$ is elementary abelian.
    As $G$ is two-generated, the largest elementary abelian quotient of $G$, which coincides with $G/\Frat(G)$, has rank~$2$.
    As $G/Z(G)$ can not be cyclic, as $G$ is not abelian, we conclude that indeed $G/Z(G)$ has order $p^2$.
\end{proof}

We next record two $F$-invariants.
The first is also recorded in~\cite[Corollary 1]{Baginski88} and though it is stated there over prime fields, the proof works over any field.
We include a proof however, to illustrate how to use \cref{prop:TransferLemmaApplications}.

\begin{prop}\label{prop:CyclicG'Invariant}
    Assume $G'$ is cyclic.
    Then its isomorphism type is an $F$-invariant.
\end{prop}
\begin{proof}
    By the fact that $\Delta(G')FG$ equals the ideal generated by the Lie commutators, $G'$ has the properties needed to apply \cref{prop:TransferLemmaApplications}.
    Hence the isomorphism type of $D_1(G')/D_2(G')$ is an $F$-invariant and this group is cyclic if and only if $G'$ is cyclic, as it is just an elementary abelian group of rank equal to the size of a minimal generating set of $G'$.
    So if $FG \cong FH$, then $H'$ is also cyclic.
    Moreover, as $G/G' \cong H/H'$ by \cref{prop:BasicInvariants}, the derived subgroups have the same order.
    So $G' \cong H'$.
\end{proof}

\begin{prop}\label{prop:GModAgemoG'}
    Assume $G'$ is cyclic and $G$ is two-generated.
    Then the isomorphism type of $G/\Agemo(G')$ is an $F$-invariant.
\end{prop}
\begin{proof}
    Assume $FG \cong FH$.
    The statement follows from \cref{th:KnownClasses} if $G$ is abelian, so assume $G$ is non-abelian.
    Observe that $\Delta(G')FG$ equals the ideal generated by the Lie commutators and $(\Delta(G')FG)^p = \Delta(\Agemo(G'))FG$.
    As $H'$ is also cyclic by \cref{prop:CyclicG'Invariant}, the ideal $\Delta(\Agemo(G'))FG$ is mapped to the ideal $\Delta(\Agemo(H'))FH$ under the isomorphism $FG \to FH$ and thus
    \begin{equation}\label{eq:IsoForGmodAgemo}
        F[G/\Agemo(G')] \cong FG/\Delta(\Agemo(G'))FG \cong FH/\Delta(\Agemo(H'))FH \cong F[H/\Agemo(H')].
    \end{equation}
    As $G$ is non-abelian two-generated, so is $G/\Agemo(G')$.
    Hence, by \cref{lem:G'HasOrderP} the center of the group $G/\Agemo(G')$ has index $p^2$ and $G/\Agemo(G')$ lies in a class for which MIP over all fields holds by \cref{th:KnownClasses}.
    Thus by \eqref{eq:IsoForGmodAgemo} we conclude $G/\Agemo(G') \cong H/\Agemo(H')$.
\end{proof}

\begin{thm}\label{thm:Metacyclic}
    Let $G$ be a finite metacyclic $p$-group and $F$ a field of characteristic~$p$.
    If $FG \cong FH$ for some group $H$, then $G \cong H$.
\end{thm}
\begin{proof}
    Since $G$ is metacyclic, it is two-generated and $G'$ is cyclic.
    By~\cref{prop:GModAgemoG'} we have $G/\Agemo(G') \cong H/\Agemo(H')$.
    Applying \cref{lem:MetacyclicModFratG'} yields that $H$ is also metacyclic.

    From here on, we can follow Sandling's proof of \cite[Theorem 5]{Sandling96} verbatim, as he only uses the $F$-invariants $G/G'$, $Z(G)$ and the number of conjugacy classes which are squares for $p=2$, together with the fact the MIP holds for $2$-groups of maximal class.
    All of these arguments can hence be repeated over all fields by \cref{prop:BasicInvariants} and \cref{prop:InvariantsCC}.
    We note that Sandling relied on a classification of metacyclic $p$-groups given in \cite{King73} which was noted to have some flaws in \cite[Remark 3.4]{XuZhang06}.
    However, checking classifications of metacyclic $p$-groups which appeared later, such as \cite{XuZhang06} or \cite{Hempel00}, we see that the relevant \cite[Theorem 3.3]{King73} remains correct.
    The verification of this fact is elementary though tedious.
    Hence Sandling's proof still works and the result follows.
\end{proof}

\subsection{Finite $3$-groups of maximal class}
In this section we reprove the main result from a paper of Bagi\'nski and Kurdics~\cite{BaginskiKurdics19} generalizing the coefficient field.
Most probably their proof works actually over any field, but we present a proof of this fact which is shorter and also uses an $F$-invariant rarely used in manual calculations before.
Namely, this is to our knowledge only the second time, after \cite{BaginskiKonovalov04}, that the $F$-invariant $\dim HH^1(FG)$, i.e.\ the dimension of the first Hochschild cohomology group of the group algebra, is applied in explicit calculations (it has been a part of computer aided investigations already since \cite{Wursthorn93}).

We call a group of order $3^n$ of maximal class if it has nilpotency class $n - 1$ and $n \geq 4$ for convenience.
We write the groups of order $3^n$ and of maximal class as quotients of the following infinite group
\[
    \mathcal{G}_n = \gen{ a, b, c, d \given c = [b, a], \ d = [c, a], \ [d, a] = c^{-3}d^{-3}, \ [d, b] = [d, c] = 1 }
\]
where additionally the following relations hold in $\mathcal{G}_n$:
\[
    \begin{cases}
        c^{3^{(n-2)/2}} = d^{3^{(n-2)/2}} = 1 & (\text{$n$ is even}), \\
        c^{3^{(n-1)/2}} = d^{3^{(n-3)/2}} = 1 & (\text{$n$ is odd}).
    \end{cases}
\]

Set
\[
    z =
    \begin{cases}
        d^{(-3)^{(n-4)/2}} & (\text{$n$ is even}), \\
        c^{(-3)^{(n-3)/2}} & (\text{$n$ is odd}).
    \end{cases}
\]

We now define the seven series of groups of interest as quotients of $\mathcal{G}_n$ by additional relations as given in~\cref{tbl:Max3Groups}.
The series of groups $T_1$, \dots, $T_7$ are the same as in the original paper, and in the notation of~\cite{BaginskiKurdics19} isomorphisms from our groups to the original groups are given explicitly as
\[
    a \mapsto s,\quad b \mapsto s_1,\quad c \mapsto s_2,\quad d \mapsto s_3,\quad z \mapsto s_{n - 1}.
\]

\begin{table}[h]
    \centering
    \begin{tabular}{lccccccc}
        \toprule
                & $T_1$          & $T_2$           & $T_3$                & $T_4$          & $T_5$          & $T_6$          & $T_7$          \\
        \midrule
        $a^3$   & 1              & 1               & 1                    & $z$            & 1              & $z$            & $z^{-1}$       \\
        $b^3$   & $c^{-3}d^{-1}$ & $c^{-3}d^{-1}z$ & $c^{-3}d^{-1}z^{-1}$ & $c^{-3}d^{-1}$ & $c^{-3}d^{-1}$ & $c^{-3}d^{-1}$ & $c^{-3}d^{-1}$ \\
        $[c,b]$ & 1              & 1               & 1                    & 1              & $z^{-1}$       & $z^{-1}$       & $z^{-1}$       \\
        \bottomrule
    \end{tabular}
    \caption{Distinct relations of the $3$-groups of maximal class.}\label{tbl:Max3Groups}
\end{table}

We note that $T_2$ and $T_3$ are isomorphic, if $n$ is odd.
Moreover, the groups $T_5$, $T_6$ and $T_7$ are only defined if $n \geq 5$.

Note that $Z(G) = \langle z \rangle$.
Throughout this section we write \[ \text{$N = \langle b, c, d \rangle$ and $M = \langle c^3, d \rangle$} \] for short.
The conjugacy classes and centralizers of $3$-groups of maximal class can be described as in \cref{tbl:ConjugacyClassesT1-T4,tbl:ConjugacyClassesT5-T7}.
\begin{table}[htbp]
    \centering
    \begin{tabular}{lccc}
        \toprule
        $E$                         & $Z(G)$ & $N \setminus Z(G)$ & $G \setminus N$           \\
        \midrule
        Number of elements          & $3$    & $3^{n-1} - 3$      & $3^n - 3^{n-1}$           \\
        Number of conjugacy classes & $3$    & $3^{n-2} - 1$      & $6$                       \\
        Class length                & $1$    & $3$                & $3^{n-2}$                 \\
        $|C_G(g)| \ (g \in E)$      & $3^n$  & $3^{n-1}$          & $9$                       \\
        $C_G(g)   \ (g \in E)$      & $G$    & $N$                & $\langle g, Z(G) \rangle$ \\
        \bottomrule
    \end{tabular}
    \caption{Conjugacy classes and centralizers of $G = T_1, \dots, T_4$.}\label{tbl:ConjugacyClassesT1-T4}
\end{table}
\begin{table}[htbp]
    \centering
    \begin{tabular}{lcccc}
        \toprule
        $E$                         & $Z(G)$ & $M \setminus Z(G)$ & $N \setminus M$        & $G \setminus N$           \\
        \midrule
        Number of elements          & $3$    & $3^{n-3} - 3$      & $3^{n-1} - 3^{n-3}$    & $3^n - 3^{n-1}$           \\
        Number of conjugacy classes & $3$    & $3^{n-4} - 1$      & $3^{n-3} - 3^{n-5}$    & $6$                       \\
        Class length                & $1$    & $3$                & $9$                    & $3^{n-2}$                 \\
        $|C_G(g)| \ (g \in E)$      & $3^n$  & $3^{n-1}$          & $3^{n-2}$              & $9$                       \\
        $C_G(g)   \ (g \in E)$      & $G$    & $N$                & $\langle g, M \rangle$ & $\langle g, Z(G) \rangle$ \\
        \bottomrule
    \end{tabular}
    \caption{Conjugacy classes and centralizers of $G = T_5, \dots, T_7$.}\label{tbl:ConjugacyClassesT5-T7}
\end{table}

\begin{prop}\label{prop:RoggenkampParameter}
    Let $G$ be a group of order $3^n$ and of maximal class.
    Then the dimension of the first Hochschild cohomology group of $FG$ is as follows:
    \[
        \dim HH^1(FG) =
        \begin{cases}
            16 + 2 \cdot 3^{n-2}  & (G \cong T_1),                         \\
            12 + 2 \cdot 3^{n-2}  & (\text{$G \cong T_2$ and $n \neq 4$}), \\
            38                    & (\text{$G \cong T_2$ and $n = 4$}),    \\
            12 + 2 \cdot 3^{n-2}  & (G \cong T_3),                         \\
            10 + 2 \cdot 3^{n-2}  & (G \cong T_4),                         \\
            12 + 22 \cdot 3^{n-5} & (G \cong T_5),                         \\
            10 + 22 \cdot 3^{n-5} & (G \cong T_6),                         \\
            14 + 22 \cdot 3^{n-5} & (G \cong T_7).                         \\
        \end{cases}
    \]
\end{prop}

\begin{proof}
    We will rely on the fundamental formula \[ \dim HH^1(FG) = \sum_{g^G} \dim H^1(C_G(g), F) = \sum_{g^G} \log_p |C_G(g)/\Frat(C_G(g))|, \] where the sums run over all the conjugacy classes $g^G$ of $G$, cf.~\cref{prop:CohomInvariants}.
    Hence we have to count the minimal number of generators of $C_G(g)$ for each $g \in G$.
    While in general this might need many case distinctions, it turns out that in the class studied here it is a rather short and simple calculation.

    There are three types of conjugacy classes for the groups $T_1$, $T_2$, $T_3$ and $T_4$, i.e.\ those with a maximal subgroup which is abelian, and one more type of class for the groups $T_5$, $T_6$ and $T_7$.
    We go through these types one by one.

    \case{Type 1: $C_G(g) = G$}
    In each case $Z(G) = \langle z \rangle \cong C_3$.
    So there are three such conjugacy classes and their centralizer $G$ is two-generated.
    The contribution to $\dim HH^1(FG)$ is hence $3 \cdot 2 = 6$ in all cases.

    \case{Type 2: $C_G(g) = N$}
    The class length is $3$.
    We always have $b^3, c^3 \in \Frat(C_G(g))$.
    If $n \geq 5$, then the formulas for $z$ also imply $z \in \Frat(C_G(g))$ and for $n=4$ we have $d = z$.
    Looking on all the cases for $b^3$ we see that $d \in \Frat(C_G(g))$, except for $n=4$ and $G \cong T_2$ in which case $b^3 = c^{-3}$.
    Hence $C_G(g)$ is always two-generated, except for $n=4$ and $G \cong T_2$ when it is three-generated.
    In $T_1$, \dots, $T_4$ the elements studied in this case are those in the maximal subgroup $N \setminus Z(G)$, i.e.\ there are $3^{n-1}-3$ elements, giving $3^{n-2}-1$ classes.
    In $T_5$, $T_6$ and $T_7$ the elements are those of $M \setminus Z(G)$, i.e.\ there are $3^{n-3}-3$ elements giving $3^{n-4}-1$ classes.
    Overall the contribution to $\dim HH^1(FG)$ is as follows:
    \[
        \begin{cases}
            2 \cdot (3^{n-2}-1) & (\text{$G \cong T_1, T_3, T_4$, or $G \cong T_2$ and $n \neq 4$}), \\
            3 \cdot (3^{n-2}-1) & (\text{$G \cong T_2$ and $n = 4$}),                                \\
            2 \cdot (3^{n-4}-1) & (G \cong T_5, T_6, T_7).
        \end{cases}
    \]

    \case{Type 3: $C_G(g) = \langle g, M \rangle$}
    This comes up only for the groups $T_5$, $T_6$ and $T_7$ and in those it is the property of the elements in $N \setminus M$.
    So there are $3^{n-1}-3^{n-3}$ elements and as each class has length $9$, this gives $3^{n-3} - 3^{n-5}$ classes.
    We claim that in all cases $C_G(g)$ is two-generated.
    To see this, let $g = b^j c^k h$ be a representative of such a class with $h \in M$ and $0 \leq j, k < 3$, but $jk \neq 0$.
    Then, as $b^3 = c^{-3}d^{-1}$, we get $g^3 = c^{3(k-j)}d^{-j} h^3$.
    As $g^3, h^3 \in \Frat(C_G(g))$, this means $c^{3(k-j)}d^{-j} \in \Frat(C_G(g))$.
    So $c^3$ and $d$ are linearly dependent modulo $\Frat(C_G(g))$, implying that $C_G(g)$ can be generated by fewer than three elements.
    Overall this type of elements hence gives a contribution to $\dim HH^1(FG)$ of exactly $2 \cdot (3^{n-3} - 3^{n-5})$.

    \case{Type 4: $C_G(g) = \langle g, Z(G) \rangle$}
    This covers all the elements in $G \setminus N$, i.e.\ there are $3^n-3^{n-1}$ elements with a class length of $3^{n-2}$, which gives $6$ classes.
    Since $g^G = gG'$ representatives of these classes are given by $a$, $a^{-1}$, $ab$, $ab^{-1}$, $a^{-1}b$ and $a^{-1}b^{-1}$.
    We need to determine whether $C_G(g)$ is cyclic which is the same as deciding if $g^3 \neq 1$.
    This property is clearly the same for $g$ and $g^{-1}$, so it is enough if we pick three of the six elements up to inverses.
    Note that $(ab)^{-1} = b^{-1}a^{-1} = (a^{-1})^{b}b^{-1} = (a^{-1}b^{-1})^b$ and similarly $(ab^{-1})^{-1} = (a^{-1}b)^{b^{-1}}$.
    Hence we need only to compute the cubes of $a$, $ab$ and $ab^{-1}$.
    The first of these is included in the defining relations giving
    \[
        a^3 =
        \begin{cases}
            1         & (G \cong T_1, T_2, T_3, T_5), \\
            z^{\pm 1} & (G \cong T_4, T_6, T_7).
        \end{cases}
    \]

    For the other cubes a straightforward calculation gives
    \[
        (ab^{\pm 1})^3 = a^3 b^{\pm 3} c^{\pm 3} d^{\pm 1} [c, b]^{-1} =
        \begin{cases}
            1         & (G \cong T_1, T_7),                \\
            z^{\pm 1} & (G \cong T_2, T_3, T_4, T_5, T_6).
        \end{cases}
    \]
    Overall the contribution to $\dim HH^1(FG)$ from this type of elements is $12$ for $G \cong T_1$, $8$ for $G \cong T_2, T_3, T_5$, $6$ for $G \cong T_4, T_6$ and $10$ for $G \cong T_7$.

    The values we obtained are summarized in~\cref{tbl:dimHH1}.
    Summing up we obtain the claimed values for $\dim HH^1(FG)$ in each case.
    \begin{table}[h]
        \begin{tabular}{lcccc}
            \toprule
                              & Type 1      & Type 2                    & Type 3                            & Type 4 \\
            \midrule
            $T_1$             & $2 \cdot 3$ & $2 \cdot (3^{n - 2} - 1)$ &                                   & $12$   \\
            $T_2\ (n \neq 4)$ & $2 \cdot 3$ & $2 \cdot (3^{n - 2} - 1)$ &                                   & $8$    \\
            $T_2\ (n = 4)$    & $2 \cdot 3$ & $3 \cdot (3^{n - 2} - 1)$ &                                   & $8$    \\
            $T_3$             & $2 \cdot 3$ & $2 \cdot (3^{n - 2} - 1)$ &                                   & $8$    \\
            $T_4$             & $2 \cdot 3$ & $2 \cdot (3^{n - 2} - 1)$ &                                   & $6$    \\
            $T_5$             & $2 \cdot 3$ & $2 \cdot (3^{n - 4} - 1)$ & $2 \cdot (3^{n - 3} - 3^{n - 5})$ & $8$    \\
            $T_6$             & $2 \cdot 3$ & $2 \cdot (3^{n - 4} - 1)$ & $2 \cdot (3^{n - 3} - 3^{n - 5})$ & $6$    \\
            $T_7$             & $2 \cdot 3$ & $2 \cdot (3^{n - 4} - 1)$ & $2 \cdot (3^{n - 3} - 3^{n - 5})$ & $10$   \\
            \bottomrule
        \end{tabular}
        \caption{Summary of contributions of $\log_p |C_G(g)/\Frat(C_G(g))|$ in $\dim HH^1(FG)$ for each group and conjugacy class type.}\label{tbl:dimHH1}
    \end{table}
\end{proof}

\begin{thm}\label{thm:BaginskiKurdics}
    Let $G$ be a group of order $3^n$ and of maximal class.
    If $FG \cong FH$ for some group $H$ and $G \not\cong H$, then $n$ is even, $n \geq 6$, $G \cong T_2$ and $H \cong T_3$, up to exchange.
\end{thm}
\begin{proof}
    By \cref{prop:Class} the group $H$ is also a $3$-group of maximal class of the same order as $G$.
    As also $\dim HH^1(FG) = \dim HH^1(FH)$ by \cref{prop:CohomInvariants}, \cref{prop:RoggenkampParameter} implies $n \geq 5$ and $G \cong T_2$ and $H \cong T_3$, up to exchange.
    Since $T_2$ and $T_3$ are isomorphic if $n$ is odd, we have $n$ is even and $n \geq 6$.
\end{proof}

\subsection{Two-generated finite $p$-groups of nilpotency class two}
The MIP was studied for the class of two-generated groups of nilpotency class two by Broche and del R\'io in \cite{BrocheDelRio21}.
They obtained a positive answer for this class over prime fields, but they partly relied on kernel size techniques which are sensible to the ground field.
Here we solve the MIP for this class of groups over any field, relying in particular on two results from \cite{GarciaLucas24, GarciaLucasMargolis24} which appeared after the original paper~\cite{BrocheDelRio21}.

\begin{thm}\label{th:BrocheDelRio}
    Let $G$ be a two-generated finite $p$-group of nilpotency class two and $F$ a field of characteristic~$p$.
    If $FG \cong FH$ for some group $H$, then $G \cong H$.
\end{thm}
\begin{proof}
    We analyze the proof of \cite[Theorem 1]{BrocheDelRio21} first to see how much of it works over any coefficient field.
    The first arguments in this proof involve the $F$-invariants $G/G'$, exponent, subsequent quotients of dimension subgroups and the nilpotency class two.
    All of these are invariants over any field by \cref{prop:BasicInvariants} and \cref{prop:Class}.
    They turn out to be enough to handle the case of $p$ odd and leave us with two subcases in the case $p=2$.
    We describe these cases next, which in the notation of \cite{BrocheDelRio21} correspond to the cases $s_1 = m-1$ and $s_1 = m$, though we will use different notation.
    We do not change the presentations of the groups however, just rename the generators and the parameters.

    \case{Case 1}
    In this case only one series of pairs of groups has to be considered depending on a single parameter $m$.
    These groups are
    \begin{align*}
        G                 &= \gen[\big]{ a, b \given [b,a]^{2^m} = [[b,a],a] = [[b,a],b] = 1, \ a^{2^m} = [b,a]^{2^{m-1}}, \ b^{2^m} = [b,a]^{2^{m-1}} } \\
        \intertext{and} H &= \gen[\big]{ x, y \given [y,x]^{2^m} = [[y,x],x] = [[y,x],y] = 1, \ x^{2^m} = [y,x]^{2^{m-1}}, \ y^{2^m} = 1 }.
    \end{align*}
    We first show that for both groups the center equals the derived subgroup and hence is cyclic.
    Set $c = [b, a]$ for brevity.
    It suffices to see that $Z(G) \leq G'$ as $Z(G) \geq G'$ is evident.
    Take an arbitrary central element $g = a^i b^j c^k$ ($0 \leq i, j, k < 2^m$) of $G$.
    Since $b a^i = a^i b c^i$ from the relations, $1 = [b, g] = c^i$ shows $i = 0$.
    Similarly, $b^j a = a b^j c^j$ and $1 = [g, a] = c^j$ shows $j = 0$.
    Then $g = c^k$ and hence $Z(G) \leq G'$.
    The same arguments show that $Z(H) = H'$.
    Moreover we have $G/Z(G) \cong H/Z(H) \cong C_{2^m} \times C_{2^m}$.
    Hence we can apply \cite[Theorem 1.3]{GarciaLucasMargolis24} to conclude that $FG \not\cong FH$ (in the notation of \cite{GarciaLucasMargolis24} we have $m(G) = m(H) = 2$).

    \case{Case 2}
    In this case we have series of pairs of groups depending on two parameters $m$ and $n$ which satisfy $n > m$.
    These groups are given by
    \begin{align*}
        G                 &= \gen[\big]{ a, b \given [b,a]^{2^m} = [[b,a],a] = [[b,a],b] = 1, \ a^{2^n} = 1, \ b^{2^m} = [b,a]^{2^{m-1}} } \\
        \intertext{and} H &= \gen[\big]{ x, y \given [y,x]^{2^m} = [[y,x],x] = [[y,x],y] = 1, \ x^{2^n} = 1, \ y^{2^m} = 1 }.
    \end{align*}
    Set $c = [b, a]$ for brevity.
    Let $U = \Omega_m(G : G')$ and $V = \Omega_m(H : H')$.

    On the one hand, if $FG \cong FH$ then \[ D_t(U)/D_{t+1}(U) \cong D_t(V)/D_{t+1}(V) \] for any positive integer $t$ by \cref{th:TransferLemma,prop:TransferLemmaApplications}.
    In particular, if $FG \cong FH$ then $|D_{2^m}(U)| = |D_{2^m}(V)|$.

    On the other hand, $U = \big\langle a^{2^{n-m}}, b, c \big\rangle$ as $U/G' = \Omega_m(G/G')$ and $G' = \langle c \rangle$.
    Note that $U' = \big\langle c^{2^{n-m}} \big\rangle$ as $\big[b, a^{2^{n-m}}\big] = c^{2^{n-m}}$.
    Since $U'$ is central and has exponent at most $2^{m-1}$, Lazard's formula yields \[ D_{2^m}(U) = \Agemo_m(U) = \big\langle b^{2^m} \big\rangle \] as $\big(a^{i 2^{n-m}} b^j c^k\big)^{2^m} = b^{j 2^m}$ ($0 \leq i, j, k < 2^m$).
    Then $|D_{2^m}(U)| = 2$ by $b^{2^m} = c^{2^{m-1}}$.
    The same calculation shows $D_{2^m}(V) = \big\langle y^{2^m} \big\rangle$ and thus $|D_{2^m}(V)| = 1$ by $y^{2^m} = 1$.
    Hence $FG \not\cong FH$.
\end{proof}

\textbf{Funding, Conflicts of interests/Competing interests and Thanks:}  This work has been supported by the Ram\'on y Cajal grant of the Spanish Agencia Estatal de Investigaci\'on of the first author (reference RYC2021-032471-I). 
The authors have no relevant financial or non-financial interests to disclose.
 
We moreover thank Diego Garc\'ia-Lucas for valuable discussions and the proof of \cref{thm:Metacyclic}.
We also thank the referee for a careful reading and detailed suggestions that improved the clarity of the manuscript.

\bibliographystyle{abbrv}
\bibliography{references}


\end{document}